\documentclass{amsart}
\usepackage{bm}
\usepackage{amsthm}
\usepackage{graphicx}
\usepackage{color}
\usepackage{enumitem}

\graphicspath{{graphics/}}

\newtheorem{theorem}{Theorem}[section]

\newtheorem{lemma}[theorem]{Lemma}
\newtheorem{proposition}[theorem]{Proposition}

\theoremstyle{remark}
\newtheorem{remark}[theorem]{Remark}
\newtheorem{question}[theorem]{Question}

\theoremstyle{definition}
\newtheorem{definition}[theorem]{Definition}
\newtheorem{example}[theorem]{Example}

\DeclareMathOperator{\dist}{d}
\DeclareMathOperator{\F}{F}
\DeclareMathOperator{\BF}{BF}
\DeclareMathOperator{\B}{B}
\DeclareMathOperator{\T}{T}
\DeclareMathOperator{\W}{W}
\DeclareMathOperator{\Y}{Y}
\DeclareMathOperator{\Cut}{Cut}
\renewcommand{\L}{\operatorname{L}}
\renewcommand{\c}{\operatorname{c}}

\newlength{\mylen}
\setbox1=\hbox{$\bullet$}\setbox2=\hbox{\tiny$\bullet$}
\begin{document}

\title{Big free groups acting on $\Lambda$-trees}
\author{Brendon LaBuz}
\address{Saint Francis University, Loretto, PA 15940}
\email{blabuz@@francis.edu}
\date{}

\begin{abstract}
The set of homotopy classes of based paths in the Hawaiian earring has a natural $\mathbb R$-tree structure, but under that metric the action by the fundamental group is not by isometries. Motivated by a suggestion by James W. Cannon and Gregory R. Conner, this paper defines an $\mathbb R^\omega$-metric that does admit for an isometric action by the fundamental group. The space does not become an $\mathbb R^\omega$-tree but is $0$-hyperbolic and embeds in an $\mathbb R^\omega$-tree. 

Cannon and Conner define big free groups $\operatorname{BF}(c)$ for cardinal number $c$ which are a generalization of the fundamental group of the Hawaiian earring. They define a big Cayley graph which coincides with the set of homotopy classes of paths in the case of the Hawaiian earring. Instead of inserting real intervals to obtain the Cayley graph, we can insert $\mathbb R^c$-intervals and obtain a new $\mathbb R^c$-tree which admits an isometric action. In fact we do not need all of $\mathbb R^c$; we can insert $\mathbb Z^c$-intervals and obtain a $\mathbb Z^c$-tree. In the case of the Hawaiian earring we give a combinatorial description of the $\mathbb Z^\omega$-tree and the corresponding action.
\end{abstract}

\maketitle
\tableofcontents

\section{Introduction}

Free groups have the property that their Cayley graphs are trees. Equivalently, the word metric on free groups is $0$-hyperbolic. We can view free groups as the fundamental group of a wedge of circles. Let $c$ be an arbitrary cardinal number and $J$ be an indexing set of cardinality $c$. Set $\W(c)=\bigvee_J S^1$. Then $\pi_1(\W(c))$ is isomorphic to $\F(c)$, the free group on $c$ generators. The generators can be realized as equivalence classes of loops. For each $j\in J$ let $a_j$ denote the equivalence class of a loop in $\W(c)$ that goes once around the $j$th circle. Then $A=\{a_j\}$ is a generating set for $\pi_1(\W(c))$.

The Hawaiian earring $E$ is a space that stands opposed to the wedge of countably infinitely many circles. We can similarly define a ``generating'' set $A$ but this set will not generate $\pi_1(E)$ since a loop in $E$ may traverse infinitely many of the circles. Motivated by this situation, Cannon and Conner define the big free group $\BF(c)$ as the set of possibly infinite products of elements of a generating set of cardinality $c$ \cite{CC}. They show that if $c$ is countably infinite then $\BF(c)$ is isomorphic to $\pi_1(E)$.

Since words in $\BF(c)$ may have infinitely many letters the word metric on $\BF(c)$ is not ideal--it would take on infinite values. For the same reason the Cayley graph of $\BF(c)$ is not connected and therefore is not a tree. In this paper a new word metric on $\BF(c)$ is defined that distinguishes between different generators and takes on values in $\mathbb Z^c$. It turns out that $\BF(c)$ is $0$-hyperbolic under this metric and therefore acts on a $\mathbb Z^c$-tree by isometries.

\section{Big free groups and $\Lambda$-trees}

We start by recalling the definition of big free group.

\subsection{Big free groups}

The notion of a group defined as the set of possibly infinite products of generators has been developed by several authors. We follow the theory of Cannon and Conner \cite{CC}. See that paper for a review of other treatments.

Let $A$ be an alphabet of arbitrary cardinality $c$ and let $A^{-1}$ denote a formal inverse set for $A$. A transfinite word is any function $w:S\to A\cup A^{-1}$ where $S$ is totally ordered and $w^{-1}(a)$ is finite for all $a\in A\cup A^{-1}$. The condition that each $a\in A\cup A^{-1}$ appears only finitely many times helps to avoid the calculation $a^{\infty}=aa^{\infty}\implies a=1$ and also allows the theory to coincide with the property of the Hawaiian earring that no circle can be traversed by a path infinitely many times. Note that if $A$ is countable then $S$ is always countable as the countable union of finite sets.

Two transfinite words $w_1:S_1\to A\cup A^{-1}$ and $w_2:S_2\to A\cup A^{-1}$ are identified if there is an order preserving bijection $\phi:S_1\to S_2$ such that $w_2\circ\phi\equiv w_1$.

A theory of infinite cancellation is required. Given a totally ordered set $S$ and $s,t\in S$ let $[s,t]_S$ denote the interval $\{r\in S:s\leq r\leq t\}$. A transfinite word $w:S\to A\cup A^{-1}$ admits a cancellation $*$ if there is a subset $T$ of $S$ and an involution $*:T\to T$ such that for each $t\in T$, $[t,t*]_S=[t,t*]_T$ ($*$ is complete), $([t,t*]_T))*=[t,t*]_T$ ($*$ is noncrossing) and $w(t*)=w(t)^{-1}$ ($*$ is an inverse pairing). The restriction of $w$ to $S-T$ is a transfinite word that arises from $w$ via the cancellation $*$. The symmetric transitive closure of this relation gives an equivalence relation on transfinite words. We say a transfinite word is reduced if it admits no nonempty cancellations. Every word admits a maximal cancellation by Zorn's Lemma and the resulting word is reduced. There may be more than one maximal cancellation but the resulting word is always the same.

\begin{theorem}[\cite{CC} 3.9]
Each equivalence class of transfinite words in $\BF(c)$ contains exactly one reduced word, up to an order preserving identification.
\end{theorem}

The product of two transfinite words $w_1:S_1\to A\cup A^{-1}$ and $w_2:S_2\to A\cup A^{-1}$ is defined as the transfinite word $w_1w_2:S_1S_2\to A\cup A^{-1}$ where $S_1S_2$ is the disjoint union of $S_1$ and $S_2$ (given the obvious ordering) and $w_1w_2|S_1\equiv w_1$ and $w_1w_2|S_2\equiv w_2$. The inverse of a transfinite word $w:S\to A\cup A^{-1}$ is the word $w^{-1}:\overline S\to A\cup A^{-1}$ where $\overline S$ reverses the ordering on $S$ and $w^{-1}(s)=w(s)$ for all $s\in S$. Thus we have a group $\BF(c)$, the big free group on an alphabet of cardinality $c$. We typically represent $\BF(c)$ as the set of all reduced transfinite words.

Cannon and Conner define the big Cayley graph $\Gamma(\BF(c))$ (\cite[Section 6]{CC}) as follows. Given a reduced transfinite word $w:S\to A\cup A^{-1}$, form the Dedekind cut space $\Cut(w)=\Cut(S)$ and then insert the real open interval $(0,1)$ between adjacent points to form the ``big interval'' $I_w$. The adjacent points are of the form $(-\infty, s)$ and $(-\infty, s]$ so they correspond to the element $w(s)\in A\cup A^{-1}$. Each inserted interval is labeled by this element. Then $\Gamma(\BF(c))$ is formed by taking the disjoint union of the $I_w$ and, for each pair, identifying the largest initial segment on which all of the labels agree. There is an action of $\BF(c)$ on $\Gamma(\BF(c))$ (see Section \ref{CayleyAction} of this paper).

In the case of the countably infinite cardinal $\aleph_0$, $\Gamma(\BF(\aleph_0))$ is in one-to-one correspondence with the space of fixed endpoint homotopy classes of paths in the Hawaiian earring and this correspondence suggests a metric for $\Gamma(\BF(\aleph_0))$ where the action is by isometries\footnote{In \cite[Theorem 6.1]{CC} the authors claim there is no metric on $\Gamma(\BF(\aleph_0))$ for which the action is by isometries. They meant no $\mathbb R$-tree metric.} \cite[Lemma 2.11]{FZ}.
However, this metric loses the large scale structure of $\Gamma(\BF(\aleph_0))$ and does not make it an $\mathbb R$-tree. In fact, under a basic condition, there is no $\mathbb R$-tree metric on $\Gamma(\BF(c))$ for which the action is by isometries (see Proposition \ref{NoTree} of this paper).

Motivated by these issues, Cannon and Conner suggest a ``big metric'' for $\Gamma(\BF(c))$. Their definition uses the tree structure of $\Gamma(\BF(c))$ to find the shortest big interval between points $x,y\in\Gamma (BF(c))$. The big metric $\dist:\Gamma(\BF(c))\times\Gamma(\BF(c))\to \mathbb R_{\geq 0}^c$ counts, for each $a\in A$, the number of occurrences of $a$ and $a^{-1}$ (with fractions occurring at the ends) in that interval. Canon and Conner suggest deriving a topology from this metric by fixing $\epsilon$ neighborhoods of 0 in  $\mathbb R_{\geq 0}^c$. We follow a different path by considering $\Gamma(\BF(c))$ as a
$\Lambda$-metric space with $\Lambda=\mathbb R^c$.

\subsection{$\Lambda$-metric spaces}

The theory of $\Lambda$-metric spaces is developed in \cite{C} and that text is the reference for the facts stated in this section. Given an abelian group $\Lambda$ and a total order $\leq$ on $\Lambda$, $\Lambda$ is an ordered abelian group if for all $a,b,c\in\Lambda$, $a\leq b$ implies $a+c\leq b+c$. Given a set $X$, a $\Lambda$-metric on $X$ is a function $\dist:X\times X\to\Lambda$ such that the usual conditions are satisfied. For all $x,y,z\in X$,

\begin{itemize}
\item $\dist(x,y)\geq 0$ 

\item $\dist(x,y)=0$ if and only if $x=y$

\item $\dist(x,y)=\dist(y,x)$

\item $\dist(x,y)\leq \dist(x,z)+\dist(z,y)$.
\end{itemize}
The topology induced by the metric is defined by the basic elements $\B(x,\epsilon)=\{y\in X:\dist(x,y)<\epsilon\}$ where $\epsilon\in \Lambda$ and $\epsilon>0$.

Let $\F(c)$ be the free group on $c$ generators where $c$ is a cardinal number. Given words $w,v\in\F(c)$, the word metric counts the number of letters in the reduced form of $w^{-1}v$. Then $\F(c)$ under the word metric is a $\mathbb Z$-metric space. Notice any $\mathbb Z$-metric space is discrete since $\mathbb Z$ has a smallest positive element.

As mentioned above the word metric does not work well for $\BF(c)$ because words may contain infinitely many letters. Given $w,v\in \BF(c)$ we count, for each $a\in A$, the occurrences of $a$ and $a^{-1}$ in $w^{-1}v$ and therefore wish to define a $\mathbb Z^c$-metric. Thus we need an order on $\mathbb Z^c$. To define the order we use an order on $A$. In the case of the fundamental group of the Hawaiian earring it is natural to consider the order on $A$ when defining a metric since these elements represent circles of decreasing size. In fact we will require $A$ to be well ordered so we define the big free group $\BF(o)$ for the ordinal number $o$ of the well ordered set $A$. We start by giving a general definition for lexicographic orders.

\begin{definition}[\cite{H}]
Let $A$ be a totally ordered indexing set and for each $a\in A$ let $S_a$ be a partially ordered set. Given $(s_a),(t_a)\in\prod_A S_a$, define $(s_a)< (t_a)$ if there exists $a\in A$ with $s_a< t_a$ and $s_b=t_b$ for all $b<a$. Define $(s_a)\leq (t_a)$ if $(s_a)<(t_a)$ or $(s_a)=(t_a)$.
\end{definition}

The indexing set $A$ is required to be totally ordered so that $\leq$ is a partial order, the lexicographic order. In the case that the $S_a$ are partially ordered abelian groups (partially ordered sets that satisfy $a\leq b\implies a+c\leq b+c$) then $\prod_A S_a$ is a partially ordered abelian group. In the case that the $S_a$ are totally ordered, the product may not be totally ordered. It is totally ordered provided $A$ is well ordered.

\begin{lemma}
Let $A$ be an indexing set and for each $a\in A$ let $S_a$ be a totally ordered set. If $A$ is well ordered then the lexicographic order on $\prod_A S_a$ is a total order. The converse holds provided the sets $S_a$ have at least two elements.

\end{lemma}

\begin{proof}
Suppose $A$ is well ordered. Suppose $(s_a),(t_a)\in\prod_A S_a$. Let $B=\{a\in A:s_a\neq t_a\}$. If $B=\emptyset$ then $(s_a)=(t_a)$. Otherwise $B$ has a least element $a$ and either $s_a<t_a$ or $s_a>t_a$. For the converse, suppose the sets $S_a$ have at least two elements. If $A$ is not well ordered then there is an infinite decreasing sequence and we can construct two elements of $\prod_A S_a$ that are not comparable.
\end{proof}

Now that we have an order on $\mathbb Z^o$ we are in a position to define a $\mathbb Z^o$-metric on $\BF(o)$. Given $w,v\in\BF(o)$, define $\dist(w,v)=(n_a)$ where $n_a$ counts the number of occurrences of $a$ and $a^{-1}$ in the reduced form of $w^{-1}v$. It is obviously symmetric and positive definite. To see the triangle inequality holds, suppose $w,v,u\in\BF(o)$. Consider the reduced form of $w^{-1}v$ and the reduced form of $v^{-1}u$ concatenated. Since for each $a\in A$ the reduced form of a word has the same or fewer occurrences of $a$ and $a^{-1}$ as the unreduced form, we must have $\dist(w,u)\leq \dist(w,v)+\dist(v,u)$.

\subsection{Geodesic $\Lambda$-metric spaces}

An important example of a $\Lambda$-metric space is $\Lambda$ itself where $\dist(a,b)=|a-b|$ for $a,b\in\Lambda$ ($|a|$ is defined in the usual way). Then we can define a $\Lambda$-geodesic in a $\Lambda$-metric space $X$ as an isometry $\alpha:[a,b]_{\Lambda}\to X$ (the interval $[a,b]_{\Lambda}$ is also defined in the usual way). We can assume $a=0$ (\cite[p.8]{C}). We call the image a segment. We will sometimes refer to a $\Lambda$-geodesic as a geodesic if the context is clear. A space is called $\Lambda$-geodesic if every pair of points can be joined by a geodesic.

The free group $\F(c)$ under the word metric is $\mathbb Z$-geodesic. Given reduced words $w,v\in\F(c)$, let $l$ be the number of initial letters that $w$ and $v$ have in common (we could have $l=0$). Suppose $w$ has $n$ letters and $v$ has $m$ letters. For each $l\leq i\leq n$ let $w_i$ be the word obtained by removing the last $n-i$ letters from $w$ and for each $l\leq i\leq m$ let $v_i$ be the word obtained by removing the last $m-i$ letters from $v$. Then $\{w_n,w_{n-1},\ldots,w_l,v_{l+1},v_{l+2},\ldots, v_m\}$ is a segment with endpoints $w$ and $v$.

On the other hand, the big free group $\BF(o)$ is not $\mathbb Z^o$-geodesic provided $o>1$. Let $a$ be the first generator and suppose $\alpha:[(0,0,\ldots),(1,0,0,0,\ldots)]_{\mathbb Z^o}\to \BF(o)$ is a geodesic from $\iota$ to $a$ where $\iota$ is the empty word. But $(1,-1,0,0,\ldots)\in[(0,0,\ldots),(1,0,0,\ldots)]_{\mathbb Z^o}$ and there is no word $w\in\BF(o)$ with $\dist(\iota,w)=(1,-1,0,0\ldots)$. We will see that $\BF(o)$ embeds isometrically in a geodesic $\mathbb Z^o$-metric space and that it acts on that space by isometries.

\subsection{$\Lambda$-trees}

A $\Lambda$-tree is defined to be a $\Lambda$-geodesic $\Lambda$-metric space $X$ such that following conditions are satisfied.
\begin{enumerate}
\item If two segments in $X$ intersect in a single point, which is an endpoint of both, then their union is a segment.

\item If two segments in $X$ have a common endpoint, then their intersection is also a segment.
\end{enumerate}
In the case that $\Lambda=\mathbb Z$ or $\Lambda=\mathbb R$, condition (2) is automatically satisfied (\cite[Lemma 1.2.3]{C}). A $\Lambda$-tree is uniquely geodesic (\cite[Lemma 1.3.6]{C}) and if $X$ is a $\Lambda$-tree, for $x,y\in X$ we write $[x,y]$ to denote the unique segment between $x$ and $y$.

The definition above is formulated in terms of basic facts about classical trees. There is another characterization that relies on the concept of a metric space being $\delta$-hyperbolic. We recount the definition and then comment on the case of $\delta=0$ in relation to trees.

Let $X$ be a $\Lambda$-metric space and let $v\in X$ be a basepoint. Given $x,y\in X$, the Gromov product of $x$ and $y$ with respect to $v$ is $(x\cdot y)_v= \frac{1}{2}(\dist(v,x)+\dist(v,y)-\dist(x,y))$. We usually suppress the notation of the basepoint and write $x\cdot y$. Notice in general we may have $x\cdot y\notin \Lambda$ (it is in $\frac{1}{2}\Lambda$). However, in the case of a $\Lambda$-tree, $x\cdot y\in\Lambda$ and it measures how long the segments $[v,x]$ and $[v,y]$ coincide. For there is a $u\in X$ with $[v,x]\cap [v,y]=[v,u]$ and $[x,u]\cup [u,y]=[x,y]$ (see Figure \ref{tree_distance}). Then
\begin{align*}
x\cdot y&=\frac{1}{2}(\dist(v,x)+\dist(v,y)-\dist(x,y))\\
&=\frac{1}{2}(\dist(v,u)+\dist(u,x)+\dist(v,u)+\dist(u,y)-\dist(u,x)-\dist(u,y))\\
&=\dist(v,u).\end{align*}
Denote the point $u$ in the above argument as $\Y(v,x,y)$. It does not depend on the order of $v,x,y$.

\begin{figure}
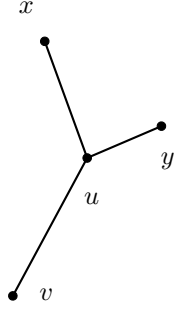
\caption{The Gromov product}
\label{tree_distance}
\end{figure}

In the case of the the free group $\F(c)$, if $w,v\in\F(c)$ and $\iota$ is the basepoint then $w\cdot v$ counts the number of initial letters that $w$ and $v$ have in common. Similarly, for $w,v\in\BF(o)$, $w\cdot v$ counts, for each $a\in A$, the number of occurrences of $a$ and $a^{-1}$ in the initial letters that $w$ and $v$ have in common.

Given $\delta\in\Lambda$ and $v\in X$, the $\Lambda$-metric space $X$ is $\delta$-hyperbolic with respect to $v$ if for all $x,y,z\in X$, $(x\cdot y)_v\geq\min\{(x\cdot z)_v,(y\cdot z)_v\}-\delta$. We say $X$ is $\delta$-hyperbolic if it is $\delta$-hyperbolic for all $v\in X$. If $X$ is $\delta$-hyperbolic with respect to one basepoint then it is $2\delta$-hyperbolic with respect to any other basepoint. Thus if a space is $0$-hyperbolic with respect to one basepoint then it is $0$-hyperbolic. 

In the case of $0$-hyperbolicity the requirement becomes $x\cdot y\geq \min\{x\cdot z,y\cdot z\}$. But we also have $x\cdot z\geq \min\{x\cdot y,y\cdot z\}$ and $y\cdot z\geq \min\{x\cdot y,x\cdot z\}$. By choosing the smallest of $x\cdot y$, $x\cdot z$, and $y\cdot z$ we see that it must be equal to one of the other two. In other words, two of $x\cdot y$, $x\cdot z$, and $y\cdot z$ are equal and they are less than or equal to the third. 

Both $\F(c)$ and $\BF(o)$ are $0$-hyperbolic. Suppose $w,v,u\in\F(c)$. Suppose without loss of generality that $w\cdot v\leq v\cdot u$. Then $v$ and $u$ have at least as many initial letters in common as $w$ and $v$ so $w$ and $u$ must have the same number of initial letters in common, that is, $w\cdot u=w\cdot v$. A similar argument shows that $\BF(o)$ is $0$-hyperbolic.

The following theorem combines \cite[Lemma 2.1.6]{C} and \cite[Lemma 2.4.3]{C}.

\begin{theorem}
Suppose $X$ is a geodesic $\Lambda$-metric space. The following statements are equivalent.

\begin{enumerate}
\item $X$ is $0$-hyperbolic and there is a basepoint $v\in X$ such that $(x\cdot y)_v\in\Lambda$ for all $x,y\in X$.

\item $X$ is a $\Lambda$-tree.
\end{enumerate}
\end{theorem}

Thus $\F(c)$ is a $\mathbb Z$-tree. Given any $\mathbb Z$-tree $X$ there is a classical tree $\Gamma$ with $X$ as the set of vertices and the $\mathbb Z$-metric of $X$ is the path metric on  $\Gamma$. In the case of $\F(c)$, $\Gamma$ is the Cayley graph.

\subsection{Groups acting on $\Lambda$-trees}

We know that $\BF(o)$ is not a $\mathbb Z^o$-tree since it is not $\mathbb Z^o$-geodesic. However there is a standard construction of a $\Lambda$-tree from a $\Lambda$-metric space that satisfies (1) in the above theorem. In the context of the space being a group we obtain an isometric action of the group on the $\Lambda$-tree. It is convenient to use the notation of a length function.

Given a group $G$ and an ordered abelian group $\Lambda$, a length function is a function $\L:G\to \Lambda$ such that the following conditions are satisfied.
\begin{enumerate}
\item $\L(g)=0$ if and only if $g=1$.

\item $\L(g)=L(g^{-1})$ for all $g\in G$.

\item For all $g,h,k\in G$, $\c(g,h)\geq\min\{\c(g,k),\c(h,k)\}$ where $\c(g,h)=\frac{1}{2}(\L(g)+\L(h)-\L(g^{-1}h))$.
\end{enumerate}
This definition is that of a Lyndon length function in \cite{C} except that only the reverse direction in condition (1) is assumed there. A length function $\L$ induces a metric $\dist$ on $G$ where $\dist(g,h)=\L(g^{-1}h)$ for $g,h\in G$ (condition (3) implies that the triangle inequality holds). Notice $\c(g,h)=(g\cdot h)_1$. Because of condition (3), $G$ is $0$-hyperbolic under $\dist$.

Given $w\in\F(c)$, let $\L(w)$ be the number of letters in the reduced word $w$. Then the induced metric is the word metric. Similarly, for $w\in\BF(c)$, set $\L(w)=(n_a)\in\mathbb Z^o$ where for each $a\in A$, $n_a$ is the number of occurrences of $a$ and $a^{-1}$ in the reduced word $w$. 

If a group $G$ has a length function such that $\c(g,h)\in\Lambda$ for all $g,h\in G$, then there is a canonical $\Lambda$-tree on which it acts by isometries. The $\Lambda$-tree $\T(G)$ is constructed in \cite[Theorem 2.4.6]{C} by taking the disjoint union of $\Lambda$-intervals $[0,\L(g)]$ for each $g\in G$ and then identifying $n\in [0,\L(g)]$ and $n\in [0,\L(h)]$ if $n\leq \c(g,h)$. Denote the equivalence class of $n\in [0,\L(g)]$ by $\langle n,g\rangle$.

To see the appropriate metric to put on $\T(G)$ let us examine  the metric of $\Lambda$-trees more closely. The following calculation is from \cite[Lemma 2.1.2(2)]{C}. Suppose $X$ is a $\Lambda$-tree with basepoint $v$. Let $x,y\in X$ and set $u=\Y(v,x,y)$ as in Figure \ref{tree_distance}. Let $x_n\in[v,x]$ be the point that is distance $n$ from $v$ and $y_m\in [v,y]$ be the point that is distance $m$ from $v$. Then $\dist(x_n,y_m)=n+m-2\min\{n,m,x\cdot y\}$. For if $n\leq x\cdot y$ then $x_n,y_m\in[v,y]$ so $\dist(x_n,y_m)=|n-m|$. A symmetric statement holds if $m\leq x\cdot y$. If $n>x\cdot y$ and $m>x\cdot y$ then $x_n,y_m\in [x,u]\cup[u,y]=[x,y]$ so
\begin{align*}\dist(x_n,y_m)&=\dist(x_n,u)+\dist(u,y_m)\\
&=(\dist(x_n,v)-x\cdot y)+(\dist(y_m,v)-x\cdot y)\\
&=n+m-2x\cdot y.\end{align*}
We use the same formula to define the metric on $\T(G)$; given $\langle n,g\rangle,\langle m,h\rangle\in\T(G)$, set $\dist(\langle n,g\rangle,\langle m,h\rangle)=n+m-2\min\{n,m,\c(g,h)\}$. Then $\T(G)$ is a $\Lambda$-tree. Notice $G$ embeds isometrically in $\T(G)$ where $g\in G$ is sent to $\langle \L(g),g\rangle$. We may just write $g\in \T(G)$ with the understanding that $g=\langle \L(g),g\rangle$.

There is an action of $G$ on $\T(G)$ by isometries where, given $\langle n,g\rangle\in\T(G)$ and $h\in G$, $\langle n,g\rangle$ is sent to the point on the segment $[h,hg]$ that is distance $n$ from $h$. 

\begin{theorem}
The action of $G$ on $\T(G)$ is by isometries.
\end{theorem}

\begin{proof}
The theorem follows from Theorems 2.4.4, 2.4.5, and 2.4.6 in \cite{C} but a direct proof is instructive. 

Let $h,g,k\in G$ and notice $\c(g,k)=(hg\cdot hk)_h$. Thus the action can be thought of as changing the basepoint from $1$ to $h$ (see Figure \ref{tree_action}). Set 
$u=\Y(h,hg,hk)$.

We first show the action of $h$ on $\T(G)$ is well defined. Suppose 
$\langle n,g\rangle=\langle n,k\rangle\in\T(G)$, so $n\leq \c(g,k)=(hg\cdot hk)_h$. Thus $h\langle n,g\rangle,h\langle n,k\rangle\in [h,u]$ so we must have $h\langle n,g\rangle=h\langle n,k\rangle$.

To see that the action is an isometry, suppose $\langle n,g\rangle,\langle m,k\rangle\in\T(G)$. If $n\leq \c(g,k)$, then $h\langle n,g\rangle,h\langle m,k\rangle\in [h,k]$ so $\dist(h\langle n,g\rangle,h\langle m,k\rangle)=|n-m|$. A symmetric statement holds if $m\leq \c(g,k)$. If $m>\c(g,k)$ and $n>\c(g,k)$ then $h\langle n,g\rangle\in[u,hg]$ and $h\langle m,k\rangle\in[u,hk]$ so $\dist(h\langle n,g\rangle,h\langle m,k\rangle)=n+m-2\c(g,h)$.

Finally for surjectivity, suppose $\langle m,k\rangle\in\T(G)$. Set $w=\Y(1,h,k)$. If $m\leq \c(h,k)$ then $\langle m,k\rangle\in[1,w]$ and $h\langle \L(h)-m,h^{-1}\rangle=\langle m,h\rangle=\langle m,k\rangle$. If $m\geq \c(h,k)$, then $\langle m,k\rangle\in[w,k]$ and $h\langle \L(h)+m-2\c(h,k),h^{-1}k\rangle=\langle m,k\rangle$.
\end{proof}

\begin{figure}
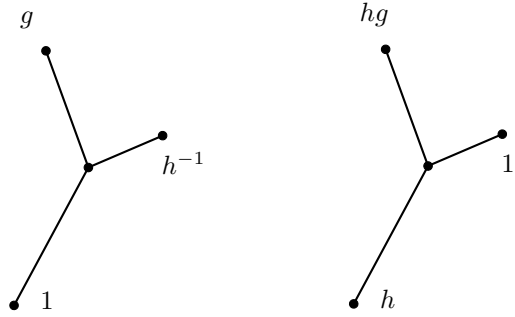
\caption{The action of $G$ on $\T(G)$}
\label{tree_action}
\end{figure}

We can read a formula for the action from Figure \ref{tree_action}. If $n\leq \c(g,h^{-1})$ then $h\cdot \langle n,g\rangle=\langle \L(h)-n,h\rangle$. If $n\geq \c(g,h^{-1})$ then $h\cdot \langle n,g\rangle=\langle \L(h)+n-2\c(g,h^{-1}),hg\rangle$. Of course there are other ways of measuring the first coordinate of $h\cdot \langle n,g\rangle$ in the second case.

The big free group $\BF(o)$ acts by isometries on the $\mathbb Z^o$-tree $\T(\BF(o))$. The action is free and without inversions.

\section{Big free groups acting on the big Cayley graph}\label{CayleyAction}

In addition to $\BF(o)$ acting on the $\mathbb Z^o$-tree $\T(\BF(o))$, we also have $\BF(o)$ acting on the big Cayley graph $\Gamma(\BF(o))$. The action can be described in  a way similar to the former action, but can also be described using a combinatorial description of $\Gamma(\BF(o))$. A description of the space of homotopy classes of paths in the Hawaiian earring is given in \cite[Example 4.15]{FZ} and can be extended to any $\Gamma(\BF(o))$. An element of $\Gamma(\BF(o))$ is an equivalence class of some $t$ in a real interval labeled by some $a^{\pm 1}\in A\cup A^{-1}$ and can be represented by a triple $(w,a^{\pm 1},t)$ where $w$ is the word that is read before the interval containing $t$. In order to obtain a unique representation, given a triple $(w,a^p,t)$, we assume that $w$ does not end in the letter $a^{-p}$. Also, if $t=0$ then no second coordinate is used and we may just write $w$. Then the action is defined as follows. Given $u\in\BF(o)$ and $(w,a^p,t)\in\Gamma(\BF(o))$, $u\cdot (w,a^p,t)=(uw,a^p,t)$ provided $uw$ does not have $a^{-p}$ as a last letter. In the case that it does, $u\cdot (w,a^p,t)=(uwa^{p},a^{-p},1-t)$.

Fischer and Zastrow show that while there is an $\mathbb R$-tree metric on $\Gamma(\BF(\aleph_0))$, there is no such metric where the action is by isometries \cite[Example 4.14]{FZ}. Their argument relies on lifts of paths via a generalized covering map of the Hawaiian earring, and they work under the assumption that the $\mathbb R$-tree metric induces a certain topology. We can extend essentially the same argument to all $\Gamma(\BF(o))$ under a natural geometric condition on the metric. The search for an $\mathbb R$-tree metric on $\BF(\aleph_0)$ is guided by the attempt to identify the big intervals with the $\mathbb R$-tree intervals. Note this identification is always impossible for uncountable alphabets for separability reasons.

\begin{proposition}\label{NoTree}
Let $o$ be an infinite ordinal number. There is no $\mathbb R$-tree metric on $\Gamma(\BF(o))$ where, for each $w\in\BF(o)$, the big interval between $\iota$ and $w$ coincides with the $\mathbb R$-tree interval $[\iota,w]$ and the action by $\BF(o)$ is by isometries.
\end{proposition}

\begin{proof}
Suppose there is such an $\mathbb R$-tree metric. Let $\{a_i\}_{i\in\mathbb N} \subset A$  and set $d_i=\dist(a_i,\iota)$. Then there are $p_i\in\mathbb N$ so that $\sum  p_id_i=\infty$. Set $w=a_1^{p_1}a_2^{p_2}a_3^{p_3}\cdots$. Then $[\iota,a_1^{p_1}]\cup[a_1^{p_1},a_1^{p_1}a_2^{p_2}] \cup[a_1^{p_1}a_2^{p_2},a_1^{p_1}a_2^{p_2}a_3^{p_3}]\cup\cdots\subset [\iota,w]$ and the intervals in the union are disjoint except for at endpoints by construction. Since the action is by isometries, $\sum  p_id_i\leq \L(w)$, a contradiction.
\end{proof}

The question still remains if there is any $\mathbb R$-tree metric on $\Gamma(\BF(o))$ that makes the action isometric. The following question may be more interesting.

\begin{question}
Does $\BF(o)$ act on any $\mathbb R$-tree by isometries?
\end{question}

In \cite{C} a group that acts by isometries freely and without inversions on some $\Lambda$-tree is called $\Lambda$-free. It is called tree-free if it is $\Lambda$-free for some $\Lambda$. There are the following inclusions: $\mathbb R$-free groups $\subset$ locally fully residually free groups $\subset$ tree-free groups. 
Note $\BF(o)$ is better than locally fully residually free--it is locally free.

The $\mathbb R^o$-metric that Cannon and Conner describe for $\Gamma(\BF(o))$ is an extension of the $\mathbb Z^o$-metric on $\BF(o)$. Given $(w,a^p,t),(v,b^q,s)\in\Gamma(\BF(o))$, define $\dist((w,a^p,t),(v,b^q,s))=\L(w)+t\L(a)+\L(v)+s\L(b)-2\min\{\L(w)+t\L(a),\L(v)+s\L(b),\c(wa,vb)\}$. In the case that $\c(wa,vb)\leq \L(w)$ and $\c(wa,vb)\leq \L(v)$, the formula becomes $\L(a^{-1}w^{-1}vb)-(1-t)\L(a)-(1-s)\L(b)$. In the case that $\L(w)+t\L(a)\leq \c(wa,wv)$ (which implies $wa$ is in the big interval $I_{vb}$), the formula becomes $\L(v)+s\L(b)-\L(w)-t\L(a)$. A similar formula holds if $\L(v)+s\L(b)\leq \c(wa,vb)$. Under this $\mathbb R^o$-metric, $\Gamma(\BF(o))$ is $0$-hyperbolic and the action of $\BF(o)$ on $\Gamma(\BF(o))$ is by isometries.

\begin{remark}
Since $\Gamma(\BF(o))$ is $0$-hyperbolic and $x\cdot y\in\mathbb R^o$ for all $x,y\in\Gamma(\BF(o))$, it embeds in an $\mathbb R^o$-tree and the action of $\BF(o)$ on $\Gamma(\BF(o))$ extends to to an isometric action on that $\mathbb R^o$-tree. Also, using the inclusion $\mathbb Z^o\hookrightarrow\mathbb R^o$ we obtain an isometric embedding of $\T(\BF(o))$ into an $\mathbb R^o$-tree. While this $\mathbb R^o$-tree differs from the one mentioned just before, in both of them a word $w\in\BF(o)$ is associated with an $\mathbb R^o$-interval $[(0,0,\ldots),\L(w)]_{\mathbb R^o}$.  Given a letter $a\in A$, the embedding of the real interval between $w$ and $wa$ in $\Gamma(\BF(o))$ into the $\mathbb R^o$-interval $[(0,0,\ldots),\L(wa)]$ is given by $\L(w)+t\L(a)$ for $0\leq t\leq 1$. The embedding of the $\mathbb Z^o$-interval between $w$ and $wa$ in $\T(\BF(o))$ into the same $\mathbb R^o$-interval is given by $\L(w)+t$ where $t\in[(0,0,\ldots),\L(a)]_{\mathbb Z^o}$. Thus the two embeddings intersect only at the endpoints.
\end{remark}

\section{A combinatorial description}

We wish to give a combinatorial description of $\T(\BF(o))$ similar to the one given above for $\Gamma(\BF(o))$. That is, given $\langle n,w\rangle\in\T(\BF(o))$, we wish to find $v\in \BF(o)$ and $a\in A$ so that $\langle n,w\rangle\in[v,va^{\pm 1}]$. We will see that such a description may not exist for $o>\omega$, the first infinite ordinal, and show that one does exist for $o=\omega$.

We will need the concept of a subword of a word in $\BF(o)$. First note that for any $\T(G)$, $h\in[1,g]$ if and only if $\L(h)+\L(h^{-1}g)=\L(g)$. The reverse direction is a direct calculation, and if $h\in[1,g]$ then $\L(h)\leq \c(g,h)=\frac{1}{2}(\L(g)+\L(h)-\L(g^{-1}h))$ so $\L(h)+\L(h^{-1}g)\leq \L(g)$. But we always have $\L(h)+\L(h^{-1}g)\geq\L(g)$ by the triangle inequality. Now consider $v,w\in\BF(o)$. The equation $\L(v)+\L(v^{-1}w)=\L(w)$ means that there is no reduction in the product of $v$ and $v^{-1}w$. In that case if we write $v:S_v\to A\cup A^{-1}$ and $w:S_w\to A\cup A^{-1}$ then we may assume $S_v\subset S_w$ and it follows that $S_v$ is a Dedekind cut of $S_w$. We call such a $v$ a subword of $w$. Cannon and Conner note that there is a one-to-one correspondence between the Dedekind cuts of $S_w$ and the subwords of $w$. This correspondence induces a linear order on the set of subwords of $w$.

If $\langle n,w\rangle\in [v,va^p]$, then $v$ is a subword of $w$.  For then $\L(v)\leq n\leq \c(w,va^p)$ and if $v^{-1}w$ does not begin with $a^p$, $\c(w,va^p)=\c(v,w)$ so $\L(v)=\c(v,w)$. If $v^{-1}w$ does begin with $a^p$ then $\c(w,va^p)=\c(v,w)+\L(a)$ and since two of $\c(w,va^p)$, $\c(v,w)$, and $\c(v,va^p)$ are equal and not greater than the third, we must have $\c(v,w)=\c(v,va^p)=\L(v)$.

We now show that if $o>\omega$ then there is $\langle n,w\rangle\in\T(\BF(o))$ such that there are no $v\in \BF(o)$ and $a\in A$ with $\langle n,w\rangle\in [v,va^{\pm 1}]$. We assume $o=\omega+1$ for convenience of notation but the argument extends to any $o>\omega$.

\begin{example}
Set $o=\omega +1$, say $A=\{a_1,a_2,\ldots,b\}$. Define the word $w=\cdots a_3a_2a_1$. Suppose there is a subword $v$ of $w$ and $a\in A$ so that $\langle \L(b),w\rangle\in [v,va^{\pm 1}]$.  Then $\L(v)\leq \L(b)$ and the only such $v$ is $\iota$. Thus $\langle \L(b),w\rangle\in[\iota,a^{\pm 1}]\cap [\iota,w]$ so $\L(b)\leq \c(a^{\pm 1},w)=0$, a contradiction.

\end{example}

The situation in the above example is that there is not a first subword of $w$ that has length at least $\L(b)$. We now show that in $\BF(\omega)$ there are no problem elements like $\langle\L(b),w\rangle$; given $\langle n,w\rangle$ we can always find the first subword of $w$ that has length at least $n$.

\begin{proposition}
Given  $\langle n,w\rangle\in\T(\BF(\omega))$, there are $v\in\BF(\omega)$ and $a\in A$ so that  $\langle n,w\rangle\in [v,va^{\pm 1}]$.
\end{proposition}

\begin{proof}
We may assume $n>0$ since $\langle 0,w\rangle=\iota$. Let $a_1\in A$ be the first index where the coordinate of $n$ has nonzero value. Note this value is positive since $n>0$. There are only finitely many $b\leq a_1$ and only finitely many occurrences of each $b$ and $b^{-1}$ in $w$ so we may list them in the order induced by the order on subwords of $w$. Notice any occurrence of a $b$ or $b^{-1}$ for $b<a_1$ will cause the corresponding subword to have length larger than $n$. We look for the first instance of one of two conditions: (1) an occurrence of a $b$ or $b^{-1}$ for $b<a_1$ or (2) an occurrence of $a_1$ or $a_1^{-1}$ that causes the length of the corresponding subword to equal $n$ in the $a_1$ coordinate. 

Under the first condition we are done; let $w(s)$ be the occurrence of the $b$ or $b^{-1}$ and take $v$ to be $w$ restricted to $(-\infty,s)$. Suppose the second condition is met and let $u$ be the resulting subword. If $\L(u)\geq n$ we are done also. Otherwise let $a_2$ be the first index where the coordinate of the length of $u$ is less than the coordinate of $n$. Note $a_2>a_1$. We follow the procedure above, starting by listing the occurrences of $b$ and $b^{-1}$ for $b\leq a_2$. If the process never terminates we find $\L(w)=n$.
\end{proof}

Thus we can represent any element of $\T(\BF(\omega))$ as a triple $(w,a^p,t)$ where $w\in\BF(\omega)$, $a\in A$, $p=\pm1$, and $t\in[0,\L(a))$ (the triple represents the element $\langle \L(w)+t,wa^p\rangle$). We have restrictions analogous to those for $\Gamma(\BF(\omega))$ and can give a description of the action and distance function in the same fashion as well. Given $u\in \BF(\omega)$ and $(w,a^p,t)\in \T(\BF(\omega))$,  $u\cdot(w,a^{p},t)=(uw,a^{p},t)$ unless $uw$ ends in $a^{-p}$ in which case $u\cdot(w,a^{p},t)=(uwa^p,a^{-p},\L(a)-t)$. Given $(w,a^p,t),(v,b^q,s)\in\T(\BF(o))$, $\dist((w,a^p,t),(v,b^q,s))=\L(w^{-1}v)+t+s$ unless $w=v$ and $a^p=b^q$ in which case we simply have $|t-s|$.

Now we see that the above combinatorial representation is unique. Suppose $(w,a^p,t)$ and $(v,b^q,s)$ represent the same element in $\T(\BF(\omega))$. Then $\L(w)+t=\L(v)+s\leq \c(wa^p,vb^q)$ so $t+s+\L(b^{-q}v^{-1}wa^p)\leq \L(a)+\L(b)$. Thus we must have $w=v$ since otherwise $\L(b^{-q}v^{-1}wa^p)>\L(a)+\L(b)$. Therefore $t=s$ and $t+s+\L(b^{-q}a^p)\leq \L(a)+\L(b)$. If $t=s>0$ then we must have $a^p=b^q$. If $t=s=0$ then we do not use a second coordinate.

As an application of our combinatorial description, we describe the quotient of the tree $\T(\BF(\omega))$ under the action of $\BF(\omega)$ as the wedge of $\mathbb Z^\omega$-circles $C_a$ for $a\in A$. Given $a\in A$, let $C_a$ be the $\mathbb Z^\omega$-interval $[0,\L(a)]$ with the endpoints identified. Define the distance between points $s$ and $t$ in the circle to be $\min\{|s-t|,\L(a)-|s-t|\}$. The formula works with both $0$ and $\L(a)$ used for the identification point.

The quotient is the set $\T(\BF(\omega))$ under the identification of elements that are in the same orbit. Thus we map an element $(w,a^p,t)$ of the quotient to the point $t\in C_a$ if $p=1$ and the point $\L(a)-t$ if $p=-1$. We show that this mapping is a bijection. 

Let $(w,a^p,t)\in\T(\BF(\omega))$ and $u\in\BF(\omega)$. Then $u\cdot (w,a^p,t)=(uw,a^p,t)$ unless $uw$ ends in $a^{-p}$ in which case $u\cdot (w,a^p,t)=(uwa^p,a^{-p},\L(a)-t)$. First suppose $p=1$. In the first case both $(w,a,t)$ and $u\cdot (w,a,t)$ are sent to $t\in C(a)$. In the second case $u\cdot (w,a,t)=(uwa,a^{-1},\L(a)-t)$  is also sent to $t\in C(a)$. If $p=-1$ we find that all of the elements are sent to $\L(a)-t$. Thus the mapping is well defined. Now suppose $(w,a^p,t)$ and $(v,a^q,s)$ are mapped to the same point $r\in C(a)$.
If $p=q=1$ we obtain $t=s$ and $(v,a,t)=vw^{-1}\cdot (w,a,t)$ since $vw^{-1}w=v$ does not end in $a^{-1}$ by assumption. If $p=1$ and $q=-1$ we have $s=\L(a)-t$ and $(v,a^{-1},\L(a)-t)=va^{-1}w^{-1}\cdot (w,a,t)$ since $va^{-1}w^{-1}w=va^{-1}$ does end in $a^{-1}$ ($v$ does not end in $a$). The other cases for $p$ and $q$ are handled in a similar fashion. Finally, given $t\in C_a$ we have $(\iota,a,t)$ in the quotient space which is sent to $t$.

We define a $\mathbb Z^\omega$-metric on the quotient space. Given $(w,a^p,t)$ and $(v,b^q,s)$, define the distance in the quotient space to be $t+s$ unless $w=v$ and $a^p=b^q$ in which case we set the distance to be $|t-s|$. It is the quotient metric inherited from $\T(\BF(\omega))$ and it is easy to check that it is definite. If the wedge $\bigvee C_a$ is given the standard wedge metric then the mapping is an isometry.

\section{The induced topology}

The fact that $\BF(\omega)$ is isomorphic to the fundamental group of the Hawaiian earring suggests a topology for it that is inherited from a standard topology on the space of fixed endpoint homotopy classes of paths called the whisker topology. Given a path $\alpha$ in $X$ and a neighborhood $U$ of the endpoint of $\alpha$, the basis element $\B([\alpha],U)=\{[\beta]: \beta=\alpha\gamma$ for some path $\gamma$ whose image lies in $U\}$ (see \cite{FZ}). We define the following topology on $\BF(o)$ following the the above model. Let $w\in\BF(o)$ and $a\in A$. Define $\B(w,a)=\{v:v=wu$ where $u\in\BF(o)$ has each letter greater than $a\}$. It is easy to check that the $\mathbb Z^o$-metric on $\BF(o)$ induces this topology.


\begin{thebibliography}{99}

\bibitem{CC} J.W. Cannon, G.R. Conner. {\em The combinatorial structure of the Hawaiian earring group}. Topology and its Applications 106 (2000), 225--271.

\bibitem{C} I. Chiswell. \emph{Introduction to $\Lambda$-trees}. World Scientific, Singapore, 2001.


\bibitem{FZ}  H. Fischer, A. Zastrow. {\em Generalized universal
coverings and the shape group}. Fundamenta Mathematicae 197 (2007), 167--196.

\bibitem{H} E. Harzheim. \emph{Ordered sets}, Advances in Mathematics, Vol. 7. Springer, New York, 2005.


\end{thebibliography}
\end{document}